\documentclass[11pt,a4paper]{amsart}
\usepackage{amssymb,amscd,amsmath, young, stmaryrd}
\usepackage{mathrsfs, mathdots} 
\usepackage[mathcal]{eucal} 
\usepackage{float}
\usepackage[usenames]{color}
\usepackage{soul}
  \theoremstyle{plain}
  \newtheorem{thm}{Theorem}[section]
  
  \newtheorem{pro}[thm]{Proposition}
  \newtheorem{cor}[thm]{Corollary}
  
  \newtheorem{thmABC}{Theorem}

  \newtheorem*{con*}{Conjecture}
  
  \newtheorem{con}[thmABC]{Conjecture}

  \theoremstyle{remark}

  \newtheorem{dfn}[thm]{Definition}
  \newtheorem*{acknowledgements}{Acknowledgements}

  \numberwithin{equation}{section}
  \numberwithin{table}{section}

\newcommand{\N}{\mathbb{N}}

\newcommand{\Z}{\mathbb{Z}}

\newcommand{\iset}{\mathcal{I}}

\renewcommand{\epsilon}{\varepsilon}
\newcommand{\vf}{\varphi}

\renewcommand{\phi}{\varphi}
\renewcommand{\theta}{\vartheta}

\newcommand{\SR}{S^{R^c} _N}
\newcommand{\ISR}{IS^R _N}
\newcommand{\Sr}{S_\rho}

\newcommand{\uu}{^{-1}}

\DeclareMathOperator{\nset}{Neg}
\DeclareMathOperator{\DNeg}{DNeg}
\DeclareMathOperator{\stc}{stc}
\DeclareMathOperator{\invv}{inv}
\DeclareMathOperator{\stdr}{std}
\DeclareMathOperator{\istd}{istd}
\DeclareMathOperator{\mstc}{mstc}
\DeclareMathOperator{\nneg}{neg}
\DeclareMathOperator{\sy}{\pi}
\DeclareMathOperator{\nmaj}{nmaj}
\DeclareMathOperator{\nsp}{nsp}
\DeclareMathOperator{\ndes}{ndes}
\DeclareMathOperator{\nstc}{nstc}
\DeclareMathOperator{\dstc}{dstc}
\DeclareMathOperator{\ddes}{ddes}
\DeclareMathOperator{\dmaj}{dmaj}
\DeclareMathOperator{\sta}{st}

\DeclareMathOperator{\maj}{maj}
\DeclareMathOperator{\des}{des}

\DeclareMathOperator{\code}{code}
\DeclareMathOperator{\Des}{Des}

\author{Angela Carnevale} \address{Fakult\"at f\"ur Mathematik,
  Universit\"at Bielefeld, D-33501 Bielefeld, Germany}
\email{acarneva1@math.uni-bielefeld.de}

\begin{document}
 \title{On some Euler-Mahonian distributions} \date{\today}

 \begin{abstract} 
 We prove that the pair of statistics $(\des,\maj)$ on multiset permutations is equidistributed with the pair $(\stc,\invv)$ on certain quotients of the symmetric group. We define the analogue of the statistic $\stc$ on multiset permutations, whose joint distribution with the inversions equals that of $(\des,\maj)$. We extend the definition of the statistic $\stc$ to hyperoctahedral and  even hyperoctahedral groups. Such functions, together with the Coxeter length, are equidistributed, respectively, with (ndes,nmaj) and (ddes,dmaj).
\end{abstract}
\maketitle

\thispagestyle{empty}
\section{introduction}
The first result about the enumeration of multiset permutations with respect to statistics now called \emph{descent number} and \emph{major index} is due to MacMahon. Let $\rho=(\rho_1,\ldots,\rho_m)$ be a composition of $N \in \N$. We denote by $\Sr$ the set of all permutations of the multiset $\{{1}^{\rho_1},\ldots,{m}^{\rho_m}\}$. The
  \emph{descent set} $\Des(w)$ of $w =  w_1\cdots w_N\in \Sr$ is
  $\Des(w)=\{i \in [N-1] \mid w_i>w_{i+1}\}$. The descent and major index
  statistics on $S_\rho$ are 
$$\des(w)=|\Des(w)| \quad\textup{ and }\quad\maj(w)=\sum_{i\in\Des(w)} i.$$  Then (\cite[\textsection 462, Vol.~2, Ch.~IV,
  Sect.~IX]{MacMahon/16}):
  \begin{equation}\label{eq:mmpart}
    \sum_{k\geq 0}\left(\prod_{i=1}^m\binom{\rho_j +k}{k}_q\right) x^k=\frac{\sum_{w \in \Sr}{x^{\des(w)}q^{\maj(w)}}}{\prod_{i=0}^{N}(1-xq^i)} \in \Z[q][[x]],
  \end{equation}
  where, for $n, k \in \N$ we put
  \[\binom{n}{k}_p :=\frac{[n]_p!}{[n-k]_p![k]_p !}, \qquad [n]_p:=\sum_{i=0}^{n-1} p^i, \qquad [n]_p!:=\prod_{i=1}^n [i]_p \quad\mbox{and}\quad [0]_p!:=1.\] 
 The well known result about the equidistribution, on multiset permutations, of the \emph{inversion number} with the major index, goes also back to MacMahon; Foata and  Sch{\"u}tzenberger \cite{FoataSch/78} proved that such equidistribution refines, in the case of the symmetric group, to inverse descent classes.
A pair of statistics that is equidistributed with $(\des,\maj)$ is called Euler-Mahonian.
In \cite{Skandera/01} Skandera introduced an Eulerian statistic, which he called $\stc$, on the symmetric group, and proved that the pair $(\stc,\invv)$ is Euler-Mahonian.

In this note we prove that the joint distribution of $(\stc,\invv)$ on certain quotients of the symmetric group is indeed the same as the distribution of $(\des,\maj)$ on multiset permutations; we use such result to define a statistic $\mstc$ that is Eulerian on multiset permutations and that, together with $\invv$ is equidistributed with the pair $(\des,\maj)$.

The Eulerian polynomial is (essentially) the descent polynomial on the symmetric group~$S_n$. Frobenius proved (see \cite{Frobenius/1910}) that such polynomial has real, simple, negative roots, and that $-1$ features as a root if and only if $n$ is even. Simion proved later that the descent polynomials of permutations of any multiset are also real rooted, with simple, negative roots (see \cite{Simion/84}). 
We use our first result of equidistribution to show that on the set of permutations of words in the alphabet $\{1^r, 2^r\}$, the polynomial of the joint distribution of des and maj admits, for odd $r$ a unique \emph{unitary} factor.  This factorisation, together with the one of Carlitz's $q$-Eulerian polynomial (the polynomial of the joint distribution of des and maj on the symmetric group) that we show in \cite{CarnevaleVoll/16}, may be considered a refinement of Frobenius' result, and supports a conjecture we made in \cite{CarnevaleVoll/16} and  that we translate in Section \ref{stcmulti} in terms of the joint distribution of $(\stc,\invv)$ on quotients of the symmetric group.

Generalisations of MacMahon's  result \eqref{eq:mmpart} to signed permutations  were first obtained by Adin, Brenti and Roichman in \cite{AdinBrentiRoichman/01} and to even-signed permutations by Biagioli in \cite{Biagioli/03}.
In the last section of this note we define  Eulerian statistics $\nstc$ and $\dstc$ that, together with the length, are  equidistributed, respectively, with the Euler-Mahonian pairs $(\ndes,\nmaj)$ on the hyperoctahedral group and $(\ddes,\dmaj)$ on the even hyperoctahedral group.

\section{Stc on quotients of the symmetric group and multiset permutations}\label{stcmulti}
For $n,m \in \N,$ $m\leq n$ we denote with $[n]:=\{1,\ldots,n\}$ and $[m,n]:=\{m,m+1,\ldots,n\}$.  For a permutation $\sigma \in S_n$ we use the one-line notation or the disjoint cycle notation. 

The Coxeter length $\ell$ for $\sigma\in S_n$ coincides with the inversion number $\invv(\sigma):=|\{(i,j)\in[n]\times [n] \mid i<j, \,\sigma(i)>\sigma(j)\}|$. 
Also, for a (signed) permutation $\sigma\in S_n$ (respectively, $B_n$), we let $\iset(\sigma):=\{(i,j)\in [n]\times[n] \mid \sigma(i)>\sigma(j)\}$.

It is well-known that the symmetric group $S_n$ is in bijection with the set of words $w=w_1\cdots w_n $ $\in E_n$ where 
\[E_n=\{w=w_1\cdots w_n \mid w_i \in [0,n-i], \mbox{ for  } i=1,.\ldots, n-1\}. \]
One of such bijections is the Lehmer code, defined as follows.

For  $\sigma \in S_n$,  $\code(\sigma)=c_1\cdots c_n \in E_n$  where $c_i= |\{j \in [i+1,n] \mid \sigma(i)>\sigma(j)\}|$.
The sum of the $c_i$s gives, for each permutation, the inversion number. The  statistic  $\stc$, that together with the length constitutes an Euler-Mahonian pair equidistributed with $(\des,\maj)$, is defined as follows (cf. \cite[Definition 3.1]{Skandera/01}):
$\stc(\sigma)=\sta(\code(\sigma))$, where for a word $w\in E_n$ $$\sta(w)=\max\{ r \in [n] \mid\mbox{ there exists a subsequence } w_{i_1}\cdots w_{i_r} >(r-1)(r-2)\cdots 1 \, 0\}$$
that is, the maximal length of a possible staircase subword. 

For example let $\sigma=452361 \in S_6$. Then $\code(\sigma) =331110$, $\invv(\sigma)=\sum_i c_i=9$, $\stc(\sigma)=\sta(\code(\sigma))=3$. So defined, the statistic $\stc$  constitutes an Eulerian partner for the inversions on $S_n$, cf. \cite[Theorem 3.1]{Skandera/01}.
\begin{thm}\label{thm:ska}
Let $n \in \N$. Then
	\begin{equation*}
	\sum_{w\in S_n}x^{\des(w)}q^{\maj(w)}=\sum_{w\in S_n}x^{\stc(w)}q^{\ell(w)}
	\end{equation*}
\end{thm}
Given a composition of $N$, the corresponding set of multiset permutations $\Sr$ is naturally in bijection with certain quotients and inverse descent classes of $S_N$. In particular, for $\rho=(\rho_1,\ldots,\rho_m)$ a composition of $N$, for $i=1\ldots, m-1$ we let 
\begin{equation}\label{eq:subset}
 r_i:=\sum_{k=1}^i {\rho _k}\mbox{  and }R:=\{r_i, i\in [m]\}\subseteq [N-1].
 \end{equation}
We let $\SR$ and $\ISR $  denote, respectively, the quotient  and the inverse descent class of of the symmetric group
\[\SR=\{w \in S_N \mid \Des(w)\subseteq R\}, \qquad \ISR:=\{w \in S_N\mid \Des(w^{-1})\subseteq R\}. \]

A natural way to associate a permutation to a multiset permutation is the standardisation. Givena $\rho$ a composition of $N$ and a word $w$ in the alphabet $\{{1}^{\rho_1},\ldots,{m}^{\rho_m}\}$, $\stdr(w)$ is the permutation of $S_N$ obtained substituting, in the order of appearence in $w$ from left to right,  the $\rho_1$ $1$s with the sequence $1\,2\ldots\rho_1$, the $\rho_2$ $2$s with the sequence $\rho_1+1 \ldots \rho_1+\rho_2$ and so on. So for example if $\rho=(2,3,2)$ and $w=1223132\in \Sr$, then $\stdr(w)=1346275 \in S_7$.

The following result is due to Foata and Han \cite[Propriet\' e 2.2]{FoataHan/04}.
\begin{pro}\label{thm:fh}
	Let $n \in \N$, $J\subseteq [n-1]$. Then
	\begin{equation}
	\sum_{ \substack{\{w\in S_n \mid\\ {\Des(w)=J}\}}}x^{\des(w^{-1})}q^{\maj(w^{-1})}=\sum_{ \substack{\{w\in S_n \mid\\ {\Des(w)=J}\}}}x^{\stc(w)}q^{\ell(w)}
	\end{equation}
\end{pro}
\begin{pro}\label{thm:equi}
Let $N \in \N$, $\rho$  a composition of $N$ and $R\subseteq [N-1]$ as in \eqref{eq:subset}. The pair $(\stc,\ell)$ on $\SR$ is equidistributed with $(\des,\maj)$ on $\Sr$:
\begin{equation}\label{equi}
C_{\rho}(x,q):=\sum_{w \in \Sr}{x^{\des(w)}q^{\maj(w)}}=\sum_{w \in \SR}{x^{\stc(w)}q^{\ell(w)}}
\end{equation}
\end{pro}
\begin{proof}
The standardisation $\stdr$ is a bijection between $\Sr$ and $\ISR$, and preserves  $\des$ and $\maj$, so
\[\sum_{w\in \Sr}x^{\des(w)}q^{\maj(w)}=\sum_{w\in \Sr}x^{\des(\stdr(w))}q^{\maj(\stdr(w))}=\sum_{w\in \ISR}x^{\des(w)}q^{\maj(w)}.\]
By Proposition \ref{thm:fh} the last term is the desired distribution on $\SR$:
\begin{equation*}
	\sum_{w\in \ISR}x^{\des(w)}q^{\maj(w)}=\sum_{w\in \SR}x^{\stc(w)}q^{\ell(w)}.
	\end{equation*}
\end{proof}

As an application, we prove a result  about the bivariate factorisation of the polynomial $C_\rho(x,q)$, that in \cite{CarnevaleVoll/16} is used to prove deduce analytic properties of some orbit Dirichlet series. We say that a bivariate polynomial $f(x,y)\in \Z[x,y]$ is unitary if there exist integers $\alpha, \beta\geq 0$ and $g\in \Z[t]$ so that $f(x,y)=g(x^\alpha y^\beta)$ and all the complex roots of $g$ lie on the unit circle (see also \cite[Remark 2.9]{CarnevaleVoll/16}).
\begin{pro}
Let $\rho=(r,r)$ where $r \equiv 1 \pmod 2$. Then
\begin{equation}\label{fact}
C_{\rho}(x,q)=(1+xq^{r})\widetilde{C}_{\rho}(x,q),\end{equation}
where $\widetilde{C}_{\rho}(x,q)$ has no
 unitary factor.
\end{pro}
\begin{proof}
The polynomial $C_{\rho}(x,1)$, descent polynomial of $\Sr$, has all real, simple, negative roots (cf. \cite[Corollary 2]{Simion/84}). Thus a factorisation of the form \eqref{fact} implies that $\widetilde{C}_{\rho}(x,q)$ has no
 unitary factor.
To prove \eqref{fact} we define an involution $\vf$ on $\SR$  such that, for all $w\in \SR$, $|\ell(\vf (w))-\ell(w)|=r$ and $|\stc(\vf (w))-\stc(w)|=1$.

We first show that when $\rho=(\rho_1,\rho_2)$,  the statistic $\stc$ on the corresponding quotient $S_{N}^{\{\rho_1\}}$ has a very simple description: it counts  the occurrences of elements $j\in [\rho_1 +1, N]$ in the first $\rho_1$ positions.  A permutation $w\in S_{N}^{\{\rho_1\}}$ has at most a descent at $\rho_1$, so its code is of the form $\code(w)=c_1\cdots c_{\rho_1} 0\cdots 0$, with  $0\leq c_1\leq\ldots\leq c_{\rho_1}$. The first (possibly) non-zero element of the code is exactly the number of elements of the second block for which the image is in the first block, and this coincides with the length of the longest staircase subword of the code. 

Let now $\rho=(r,r)$ and $r$ odd. For $w \in \SR$ we  let \[M_w=\{i\in [r] \mid w\uu(i)\leq r \mbox{ and } w\uu(i+r)> r \mbox{ or } w\uu(i)> r \mbox{ and } w\uu(i+r)\leq r \},\]
that is, the set of $i\in [r]$ for which $i$ and $i+r$ are not in the same ascending block. Since $r$ is odd, $M_w$ is non-empty for all $w\in \Sr$.
We then define $\vf(w)=((\iota,\iota + r)w)^{R^c}$, where $\iota:=\min\{i \in M_w\}$ and, for $\sigma \in S_N$,  $\sigma^{R^c}$ denotes the unique minimal coset representative in the quotient $\SR$. Clearly $\stc(\vf(w))=\stc(w)\pm 1$.
Suppose now that $w^{-1}(\iota)\leq r$ and $w^{-1}(\iota)>r$ (the other case is analogous). Then $$\ell(\vf(w))=\ell(w)+|\{i \in [r]\mid w(i)>\iota\}|+|\{i \in [r+1,2r]\mid w(i)<\iota+r\}|=\ell(w)+r-i+i.\qedhere$$
\end{proof}
We reformulate \cite[Conjecture B]{CarnevaleVoll/16} in terms of the bivariate distribution of $(\stc,\ell)$ on quotients of $S_n$.
\begin{con}
	Let $\rho$ be a composition of $N$ and $R\subseteq [N-1]$ constructed as in \eqref{eq:subset}. Then $C_\rho (x,q)=\sum_{w\in \SR}{x^{\stc(w)}q^{\ell(w)}}$ has a unitary factor if and only if $\rho=(\rho_1,\ldots,\rho_m)$ where $\rho_1=\ldots=\rho_m=r$ for some odd $r$ and even $m$. In this case
	\[\sum_{w\in \SR}{x^{\stc(w)}q^{\ell(w)}}=(1+xq^{\frac{rm}{2}})\widetilde{C}_\rho(x,q)\]
	for some $\widetilde{C}_\rho(x,q)\in \Z[x,q]$ with no unitary factors.
\end{con}

Proposition \ref{thm:equi} suggests a natural extension of the definition of the statistic $\stc$ to multipermutations, thus answering a question raised in \cite{Skandera/01}. 

For $w\in \Sr$, $\stdr (w) \in \ISR$. So we have a bijection between multiset permutations $\Sr$ and the quotient $\SR$
\[\istd: \Sr \rightarrow \SR,\quad \istd(w)=(\stdr (w))^{-1}\]
which is inversion preserving:  $\invv(w)=\invv(\istd(w))$.
\begin{dfn}
	Let $\rho$ be a composition of $N$. For a multiset permutation $w \in \Sr$ the \emph{multistc} is
	\[\mstc(w):=\stc(\istd(w)).\]
\end{dfn}
The pair $(\mstc,\invv)$ is equidistributed with $(\des,\maj)$ on $\Sr$, as
\[\sum_{w \in \Sr} x^{\mstc(w)}q^{\invv(w)}=\sum_{w \in \SR} x^{\stc(w)}q^{\invv(w)}=\sum_{w \in \Sr} x^{\des(w)}q^{\maj(w)},\]
which together with \eqref{eq:mmpart} proves the following theorem.
\begin{thm}
	Let $\rho$ be a composition of $N\in \N$. Then
\[ \sum_{k\geq 0}\left(\prod_{i=1}^m\binom{\rho_j +k}{k}_q\right) x^k=\frac{\sum_{w \in \Sr}{x^{\mstc(w)}q^{\invv(w)}}}{\prod_{i=0}^{N}(1-xq^i)} \in \Z[q][[x]].\]
\end{thm}

\section{Signed and even-signed permutations}
MacMahon's result \eqref{eq:mmpart} for the symmetric group (i.e. for $\rho_1=\ldots\rho_m=1$) is often present in the literature as Carlitz's identity, satisfied by the Carlitz's Eulerian polynomial $A_n(x,q):=\sum_{\sigma \in S_n} x^{\des(\sigma)}q^{\maj(\sigma)}$. 

Such result was extended, for suitable statistics, to the groups of signed and even-signed permutations. The major indices  so defined are in both cases equidistributed with the Coxeter length $\ell$. In this section we define type $B$ and type $D$ analogues of the statistic $\stc$, that together with the length satisfy these generalised Carlitz's identities.
\subsection{Eulerian companion for the length on $B_n$}
Let $n\in \N$. The hyperoctahedral group $B_n$ is the group of permutations $\sigma=\sigma_1\cdots \sigma_n$ of $\{\pm 1,\ldots,\pm n\}$ for which $|\sigma|:=|\sigma_1|\ldots |\sigma_n| \in S_n$.
For $\sigma \in B_n$, the negative set and negative statistic are  
\[\nset(\sigma)=\{i \in [n] \mid \sigma(i)<0\}\quad \nneg(\sigma)=|\nset(\sigma)|.\]
The Coxeter length $\ell$ for $\sigma$ in $B_n$ has the following combinatorial interpretation (see, for instance \cite{BjoernerBrenti/05}):
\begin{equation*}
\ell(\sigma)=\invv(\sigma)+\nneg(\sigma)+\nsp(\sigma),
\end{equation*}
where $\invv$ is the usual inversion number and $\nsp(\sigma):=|\{(i,j)\in [n]\times[n] \mid i<j,\,\sigma(i)+\sigma(j)<0\}|$ is the number of \emph{negative sum pairs}.

 In  \cite{AdinBrentiRoichman/01} an Euler-Mahonian pair of the \emph{negative} type was defined as follows.
The negative descent and negative major index are, respectively,
\begin{equation}\label{eq:nstat}
\ndes(\sigma)=\des(\sigma)+\nneg(\sigma) \quad \nmaj(\sigma)=\maj(\sigma)-\sum_{i\in\nset(\sigma)}\sigma(i).\end{equation} 
The pair $(\ndes,\nmaj)$ satisfies the following generalised Carlitz's identity, cf. \cite[Theorem 3.2]{AdinBrentiRoichman/01}.
\begin{thm}
Let $n\in \N$. Then
\begin{equation}
\sum_{r\geq 0} [r+1]^n _q x^r = \frac{\sum\limits_{\sigma \in B_n}{x^{\ndes(\sigma)}q^{\nmaj(\sigma)}}}{(1-x) \prod\limits_{i=1}^{n}{(1-x^2 q^{2i})}} \mbox{ in }\Z[q][[x]].\end{equation}	
\end{thm}

Motivated by \eqref{eq:nstat} and the well-known fact that the length in type $B$ may be also written as
\begin{equation}
\ell(\sigma)=\invv(\sigma)-\sum_{i \in \nset(\sigma)} \sigma(i),
\end{equation} we define the analogue of the statistic $\stc$ for signed permutations as follows.
\begin{dfn}
	Let $\sigma\in B_n$. Then 
\[\nstc(\sigma):=\stc (\sigma)+\nneg(\sigma).\]
\end{dfn}
\begin{thm}\label{thm:B}
	Let $n\in \N$. Then
\[\sum_{\sigma\in B_n}{x^{\nstc(\sigma)}q^{\ell(\sigma)}}=\sum_{\sigma\in B_n}{x^{\ndes(\sigma)}q^{\nmaj(\sigma)}}\]
\end{thm}
\begin{proof}
	We use essentially the same argument as in the proof of \cite[Theorem 3]{LaiPet/11}. There, the following decomposition of $B_n$ is used. 
Every permutation $\tau \in S_n$ is associated with $2^n$ elements of $B_n$, via the choice of the $n$ signs. More precisely, given a signed permutation $\sigma \in B_n$ one can consider the ordinary permutation in which the elements are in the same relative positions as in $\sigma$. We write $\sy(\sigma)=\tau$. Then
\[B_n=\bigcup_{\tau \in S_n} B(\tau)\]
where $B(\tau):=\{\sigma \in B_n \mid \sy(\sigma)=\tau\}$. So every $\sigma\in B_n$ is uniquely identified by the permutation $\tau=\sy(\sigma)$ and the choice of signs $J(\sigma):=\{\sigma(j) \mid j \in \nset(\sigma)\}.$ 

Clearly, for $\sigma \in B_n$ we have $\iset(\sigma)=\iset(\sy(\sigma))$, and thus $\stc(\sigma)=\stc(\sy(\sigma))$. So, for $\tau=\sy(\sigma)$
\[x^{\nstc(\sigma)}q^{\ell(\sigma)}=x^{\stc(\tau)}q^{\invv(\tau)}\prod_{j\in J(\sigma)} x q^j.\] The claim follows, as
\[\sum_{\sigma \in B_n} {x^{\nstc(\sigma)}q^{\ell(\sigma)}}=\sum_{\sigma \in B(\tau)}\sum_{\tau \in S_n}x^{\stc(\tau)}q^{\invv(\tau)}\sum_{J\subseteq [n]}\prod_{j\in J} x q^j=A_n(x,q)\prod_{i=1}^n {(1+xq^i)}.\qedhere\]
\end{proof}
\begin{cor}Let $n\in \N$. Then
\begin{equation}
	\sum_{r\geq 0} [r+1]^n _q x^r = \frac{\sum\limits_{\sigma \in B_n}{x^{\nstc(\sigma)}q^{\ell(\sigma)}}}{(1-x) \prod\limits_{i=1}^{n}{(1-x^2 q^{2i})}} \mbox{ in }\Z[q][[x]].\end{equation}
\end{cor}
\subsection{Eulerian companion for the length on $D_n$}
The even hyperoctahedral group $D_n$ is the subgroup of $B_n$ of signed permutations for which the negative statistic is even:
\[D_n:=\{\sigma \in B_n \mid \nneg(\sigma) \equiv 0 \pmod 2\}.\]
Also for $\sigma$ in $D_n$ the Coxeter length can be computed in terms of statistics:
\begin{equation}
\ell(\sigma)=\invv(\sigma)+\nsp(\sigma).
\end{equation}
The problem of finding an analogue, on the group $D_n$ of even signed permutations, was solved in \cite{Biagioli/03}, 
where type $D$  statistics $\des$ and $\maj$ were defined, as follows. For $\sigma \in D_n$
\begin{equation}
\ddes(\sigma)=\des(\sigma)+|\DNeg(\sigma)|\qquad
\dmaj(\sigma)=\maj(\sigma)-\sum_{i \in \DNeg(\sigma)} \sigma(i)
\end{equation}
where $\DNeg(\sigma):=\{i-1 \in [n]|\sigma(i)<-1\}$.
The following holds (cf. \cite[Theorem 3.4]{Biagioli/03}).
\begin{thm}
Let $n \in \N$. Then
\begin{equation}
\sum_{r\geq 0} [r+1]^n _q x^r = \frac{\sum\limits_{\sigma \in D_n}{x^{\ddes(\sigma)}q^{\dmaj(\sigma)}}}{(1-x)(1-xq^n) \prod\limits_{i=1}^{n-1}{(1-x^2 q^{2i})}} \mbox{ in }\Z[q][[x]].\end{equation}	
\end{thm}
\begin{dfn}
Let $\sigma \in D_n$. We set
$$\dstc(\sigma):=\stc (\sigma)+|\DNeg(\sigma)|=\stc(\sigma)+\nneg(\sigma)+\epsilon(\sigma),$$
where $$\epsilon(\sigma)=\begin{cases}
-1 \mbox{ if }\sigma^{-1}(1)<0 \\
0\:\:\; \mbox{ otherwise }.
\end{cases}$$
\end{dfn}
We now show that the statistic just defined
constitutes an Eulerian partner for the length on $D_n$, that is, the following holds.
\begin{thm}\label{thm:D}
Let $n\in \N$. Then
\[\sum_{\sigma\in D_n}{x^{\dstc(\sigma)}q^{\ell(\sigma)}}=\sum_{\sigma\in D_n}{x^{\ddes(\sigma)}q^{\dmaj(\sigma)}}.\]
\end{thm}
\begin{proof}
We use, as in \cite{Biagioli/03} the following decomposition of $D_n$. Let \begin{equation}\label{eq:asc}
T_n:=\{\alpha \in D_n \mid \des(\alpha)=0\}=\{\alpha\in D_n \mid \iset(\alpha)=\emptyset\}
\end{equation}
then $D_n$ can be rewritten as the following disjoint union:
\begin{equation}\label{eq:ddec}
D_n=\bigcup_{\tau \in S_n}\{\alpha \tau \mid \alpha \in T_n\}.
\end{equation}
For $\alpha \in T_n$ and $\tau \in S_n$ one has:
\begin{equation*}
\ell(\alpha \tau)=\ell(\alpha)+\ell(\tau)=\nsp(\alpha)+\invv(\tau),\quad  \nsp(\alpha \tau)=\nsp(\alpha),\quad \dstc(\alpha \tau)=\stc(\tau)+\nneg(\alpha)+\epsilon(\sigma),
\end{equation*}
the last one follows from the second equality in \eqref{eq:asc}. We thus have
\begin{align*}
\sum_{\sigma\in D_n}{x^{\dstc(\sigma)}q^{\ell(\sigma)}}&=\sum_{\alpha \in T_n}\sum_{\tau \in S_n}x^{\stc(\tau)+\nneg(\alpha)+\epsilon(\alpha)}q^{\ell(\alpha)+\ell(\tau)}\\
&=\sum_{\alpha \in T_n}x^{\nneg(\alpha)+\epsilon(\alpha)}q^{\nsp(\alpha)} \sum_{\tau \in S_n}x^{\stc(\tau)}q^{\invv(\tau)} \\
&=\prod_{i=1}^{n-1}(1+x q^i)  A_n(x,q)\end{align*}
for the last equality see \cite[Lemma 3.3]{Biagioli/03}.
The result follows, as
\[\sum_{\sigma\in D_n}{x^{\ddes(\sigma)}q^{\dmaj(\sigma)}}=\prod_{i=1}^{n-1}(1+x q^i)  A_n(x,q).\qedhere\]
\end{proof}
\begin{cor}Let $n \in \N$. Then
	\begin{equation}
	\sum_{r\geq 0} [r+1]^n _q x^r = \frac{\sum\limits_{\sigma \in D_n}{x^{\dstc(\sigma)}q^{\ell(\sigma)}}}{(1-x)(1-xq^n) \prod\limits_{i=1}^{n-1}{(1-x^2 q^{2i})}} \mbox{ in }\Z[q][[x]].\end{equation}
\end{cor}
\begin{acknowledgements}
 The author is partially supported
  by the German-Israeli Foundation for Scientific Research and
  Development, through grant no.~1246.
\end{acknowledgements}

\bibliographystyle{amsplain}

\end{document}